\documentclass[11pt]{article}

\usepackage{amssymb, amsmath, amsthm, graphicx}
\usepackage[left=1in,top=1in,right=1in]{geometry}

\usepackage{subfigure}

\newtheorem{problem}{Problem}

      \theoremstyle{plain}
      \newtheorem{theorem}{Theorem}[section]
      \newtheorem{lemma}[theorem]{Lemma}
                 \newtheorem{corollary}[theorem]{Corollary}

\newtheorem{conjecture}[theorem]{Conjecture}
      \theoremstyle{definition}
      \newtheorem{definition}[theorem]{Definition}

      \theoremstyle{remark}

\def\proj{\mbox{\rm proj}}

\def\ex{{\rm{ex}}}
\def\e{\epsilon}

\def\cB{{\cal B}}
\def\cC{{\cal C}}

\def\d{d^*}

\title{A survey of Tur\'an problems for expansions}
\author{Dhruv Mubayi\thanks{Department of Mathematics, Statistics, and Computer Science, University of Illinois, Chicago, IL, 60607 USA.  Research partially supported by NSF grant DMS-1300138. Email: {\tt mubayi@uic.edu}} \qquad
Jacques Verstra\"ete\thanks{Department of Mathematics, University of California, San Diego, CA, 92093 USA.  Research supported by NSF grant DMS-1362650. Email: {\tt jacques@ucsd.edu}}  }

\begin{document}

\maketitle
\medskip

\begin{abstract}
The $r$-expansion $G^+$ of a graph $G$ is the $r$-uniform hypergraph obtained from $G$ by enlarging each edge of
$G$ with a  vertex subset of size $r-2$ disjoint from $V(G)$ such that distinct edges are enlarged by disjoint subsets. Let $\ex_r(n,F)$ denote the maximum number
of edges in an $r$-uniform hypergraph with $n$ vertices not containing any copy of the $r$-uniform hypergraph $F$.
Many problems in extremal set theory ask for the determination of $\ex_r(n,G^+)$ for various graphs $G$. We survey these Tur\'an-type problems, focusing on recent developments.
\end{abstract}

\section{Introduction}

An $r$-uniform hypergraph $F$, or simply {\em $r$-graph}, is a family of $r$-element subsets of a finite set.
We associate an $r$-graph $F$ with its edge set and call its vertex set $V(F)$. Given an $r$-graph $F$,
let $\ex_r(n,F)$ denote the maximum number of edges in an $r$-graph on $n$ vertices that does not contain a copy of $F$ (if the uniformity is obvious from context, we may omit the subscript $r$).
The {\em expansion} of a graph $G$ is the $r$-graph $G^+$ with edge set $\{e \cup S_e : e \in G\}$ where
$|S_e|=r-2$, $S_e \cap V(G)=\emptyset$ for all $e \in G$, and
$S_e \cap S_{e'}=\emptyset$ for $e \ne e'$. By definition, the expansion of $G$ has exactly $|G|$ edges and $|V(G)|+(r-2)|G|$ vertices.

\medskip

Expansions include many important hypergraphs who extremal functions have been investigated, for instance the celebrated Erd\H{o}s-Ko-Rado Theorem~\cite{EKR}
 is the case of expansions of a matching of size two. In this article we survey these results focusing on recent developments and proof techniques.

A crucial parameter that determines the growth rate of
 $\ex_r(n, G^+)$ is the chromatic number of $G$. Indeed, an old result of Erd\H os~\cite{E1964} implies that for fixed $r \ge 2$, the function
 $\ex_r(n, G^+)$ has order of magnitude $n^r$ if and only if $\chi(G)>r$. As we will see in the next section, the asymptotics in this case have been determined and we say that these problems fall under the nondegenerate case. The majority of open problems in the area arise when $\chi(G) \leq r$ and we will refer to this as the degenerate case.

Along with obtaining exact extremal results, one can ask about the
structure of nearly extremal structures. The seminal result in
this direction is the stability theorem for graphs,
proved independently by Erd\H os and Simonovits (see
\cite{Sstab}). Similar theorems
for hypergraphs have been proven more  recently. For example, papers \cite{FMP, FPS2, FS,
KM, KS, KS2} each prove stability theorems for hypergraphs,
\cite{FS, KS} for the Fano plane, \cite{FPS2} for $\{123, 124,
125, 345\}$, \cite{KS2, FMP} for certain types of hypergraph triangles,  \cite{KM} for
cancellative triple systems, and also for $\{123, 124, 345\}$.  By now there are many other stability results and the list above is far from exhaustive.
Stability theorems have also proved useful to obtain exact
results. This approach, first used by Simonovits~\cite{Sstab} to determine the exact Tur\'an number of color critical graphs, was  developed and applied to hypergraph problems in the above results. The method proceeds by
first proving an approximate result, then a stability statement, and
finally uses the stability result to guarantee an exact extremal
result. All of these results applied to the nondegenerate case, i.e.
when the order of magnitude of $\ex_r(n, F)$ is $n^r$.

It was shown in~\cite{M2} that this method can be used in the degenerate case as well by solving a generalization of an old question of Katona. Since then, this approach has been used many times, for example, to solve the extremal problem for $k$-regular subsystems~\cite{MV2, K}, $k$-uniform clusters and
simplices~\cite{CLW, JPY,KM2, MRam1, MRam}
and for expansions of  acyclic graphs and cycles~\cite{KMV, KMV2}.

This survey is organized as follows.  In
Section 2 we state results for the nondegenerate case, and in Section 3 we survey the results in the degenerate case.  Section 4 describes the important proof techniques, focusing on the classical Delta-system method and a new approach that uses  shadows, random sampling, and canonical Ramsey theory. In Section 5 we will use these techniques to give short proofs of a variety of recent results in the degenerate case. Specifically, we will give proofs of the asymptotic results for $\ex_3(n, K_{s,t}^+)$ when $t>(s-1)!$, for $\ex_r(n, G^+)$   when $r \ge 4$ and $G$ is a forest, and for $\ex_3(n, C_4^+)$. Our hope is that these will  illustrate many of the important proof ideas used in this area. Section 6 lists some  open problems.

\subsection{Notation and Terminology}

For a set $V$, let $2^V$ be the set of subsets $V$, and let ${V \choose r}$ be the set of $r$-element subsets of $V$.
For a positive integer $n$, we write $[n] = \{1,2,\dots,n\}$.
A hypergraph with vertex set $V$ is a family of sets $H \subset 2^V$. If $H \subset {V \choose r}$
then $H$ is an {\em $r$-graph}. If $r = 2$, $H$ is a graph and
if $r = 3$, $H$ is a triple system. If $S \subset V$, then  the subhypergraph of $H$ induced by $S$ is $H[S] = \{e \in H : e \subset S\}$.
The {\em degree of $S$} in $H$ is $d_H(S) = |\{e \in H: S \subset e\}|$, and we write $d_H(v)$ if $S = \{v\}$.
The {\em neighborhood} of $S$ is $N_H(S) = \{e \backslash S : S \subset e \in H\}$, and we write
$N_H(v)$ if $S = \{v\}$. The subscript $H$ is suppressed when $H$ is clear from context.
For $s \in \mathbb N$, define the {\em shadow $s$-graph}
\[ \partial_sH = \left\{S \in {V\choose s} :  d_H(S) \geq 1\right\}.\]
We write $\partial H$ instead of $\partial_2 H$.
For a set $S \subset V$, the {\em trace} of $H$ on $S$ is
\[ H|_S = \{e \cap S : e \in H\}.\]
For a positive integer $n$, the {\em Tur\'{a}n Number} of an $r$-graph $F$ is
\[ \ex_r(n,F) = \max\left\{|H| : H \subset {[n] \choose r}, F \not\subset H\right\}.\]
If $ F \not\subset H$ we say that $H$ is $F$-free. An  $n$-vertex $F$-free $r$-graph $H$ is {\em extremal for} $F$ if
$|H|= \ex_r(n,F)$.

 \section{The non-degenerate case}

 We begin by observing that when $l\ge r=2$ and $G=K_{l+1}$
 $$\ex_2(n, G^+)=\ex_2(n, G)=\ex_2(n, K_{l+1}).$$ Tur\'an's Theorem~\cite{T} states that $\ex(n, K_{l+1})$ is uniquely achieved (for all $n > l$) by the Tur\'an graph, which is the complete $l$-partite graph on $n$ vertices with part sizes that differ by at most one. More generally, for any graph $G$ with $\chi(G)=l+1>2$, the Erd\H os-Stone-Simonovits Theorem~\cite{ErdosStone, ES66} determines the asymptotics of $\ex_2(n, G^+)=\ex_2(n, G)$.  In general exact results are known for very few cases; one such case is when $G$ is color critical~\cite{Sstab}.

In order to state the results for $r \ge 3$,  we must generalize the definition of the Tur\'an graph to hypergraphs.
An $r$-graph is $l$-partite if its vertex set can be partitioned
into $l$ classes, such that every edge has at most one vertex from
each class. Thus in particular, there are no edges if $l<r$. A
 complete $l$-partite $r$-graph is one where all of the allowable edges (given a vertex $l$-partition) are present. For $n, l, r \ge 1$, let $T_r(n,l)$ be the
complete $l$-partite $r$-graph on $n$ vertices with no two part
sizes differing by more than one. Thus the part sizes are
$n_i=\lfloor (n+i-1)/l\rfloor$ for $i \in [l]$.  Among all
$l$-partite $r$-graphs on $n$ vertices, $T_r(n,l)$ has the most
edges. The number of edges in $T_r(n,l)$ is
$$t_r(n,l)=\sum_{S \in {[l]\choose r}}\prod_{i \in S} n_i \qquad \hbox{ and } \qquad t_r(n,l)\sim\frac{l(l-1)\cdots (l-r+1)}{l^r}{n \choose r}$$
for fixed $l\ge r$.
Clearly $K_{l+1}^+ \not\subset T_r(n, l)$ so
$\ex_r(n, K_{l+1}^+)\ge t_r(n,l)$. The first author~\cite{Mubayi} proved that if $l\ge r$ is fixed, then
$\ex_r(n, K_{l+1}^+) < t_r(n,l)+ o(n^r)$ and conjectured that this could be further improved to an exact result for $n$ sufficiently large. This was subsequently proved by Pikhurko~\cite{Pikhurko}.

\begin{theorem} {\bf (\cite{Mubayi, Pikhurko})}\label{mp}
Fix $l \ge r \ge 2$ and let $n$ be sufficiently large. Then
$$\ex_r(n, K_{l+1}^+) = t_r(n,l).$$
Moreover equality is achieved only by the Tur\'an $r$-graph $T_r(n,l)$.
\end{theorem}

It was observed by Alon and Pikhurko~\cite{AP} that the approach applied to prove Theorem \ref{mp} can be extended to any color critical graph $G$ with $\chi(G)>r$.  More generally, a result of Erd\H os~\cite{E1964}, the supersaturation technique (see Erd\H os-Simonovits~\cite{ES}), and the asymptotic result of the first author~\cite{Mubayi} give the following.

\begin{theorem} \label{ess}
Fix integers $l\ge r \ge 2$ and a graph $G$ with $\chi(G)=l+1$.  Then
$$\ex_r(n, G^+) \sim t_r(n,l).$$
\end{theorem}
Theorem~\ref{ess} determines the asymptotics of $\ex_r(n, G^+)$ for all $G$ with $\chi(G)>r$.

The Erd\H os-Simonovits stability theorem~\cite{Sstab} states that if an $n$-vertex $K_{l+1}$-free
graph ($n$ large) has almost as many edges as $T_2(n,l)$, then its
structure is very similar to that of $T_2(n,l)$. The analogous result for the $r$-expansion of $K_{l+1}$ was proved in~\cite{Mubayi}.

\begin{theorem}  {\bf (\cite{Mubayi})}\label{stab} Fix $l \ge r \ge 2$, and $\delta >0$. Then there
exists an $\e >0$ and an $n_0$ such that the following holds for
all $n >n_0$: If ${G}$ is an $n$-vertex $K_{l+1}^+$-free $r$-graph with at least $t_r(n,l)-\e n^r$
edges, then $G$ can be transformed to $T_r(n,l)$ by adding and
deleting at most $\delta n^r$ edges.
\end{theorem}

 \section{The degenerate case}
 As mentioned earlier, most of the activity around $\ex_r(n, G^+)$ has focused on the degenerate case, i.e.,  when $\chi(G) \leq r$. In fact, many famous open problems in extremal set theory deal with very simple graphs $G$.  In this section we survey these results.    Section 3.1 deals with matchings, section 3.2 with stars, section 3.3 with paths and cycles, section 3.4 with trees, and sections 3.5--3.7 with a variety of graphs containing cycles.

\subsection{Matchings}
Let us assume that $G=M_t=tK_2$ is a matching with $t$ edges. The simplest case is $t=2$, where $G^+$ comprises two disjoint edges. Here the Erd\H os-Ko-Rado theorem determines $\ex_r(n, M_2^+)$ for all $r,n$.  Say that a hypergraph $H$ is a {\em star} if there is a vertex of $H$ that lies in all edges of $H$.

\begin{theorem} {\bf (Erd\H os-Ko-Rado~\cite{EKR})} \label{ekrthm}
Let $n,r \ge 2$. If $n < 2r$, then
$\ex_r(n, M_2^+)= {n \choose r}$ while if $n \ge 2r$, then  $\ex_r(n, M_2^+)= {n-1 \choose r-1}$.
For $n >2r$,  equality holds above only for a  star.
\end{theorem}

The general case is an old conjecture of Erd\H os.

\begin{conjecture} {\bf (Matching Conjecture, Erd\H os~\cite{E1965})} \label{erdosmatching}
Let $r, t \ge 1$ and $n \ge r(t+1)-1$. Then
$$\ex_r(n, M_{t+1}^+)= \max\left\{
{r(t+1)-1 \choose r}, {n \choose r}- {n-t \choose r}\right\}.$$
\end{conjecture}

For $r=1$, the conjecture holds trivially and for $r=2$ it was proved by Erd\H os and Gallai~\cite{EG}. For $n>n_0(r,t)$ it was proved by Erd\H os~\cite{E1965}. Bollob\'as, Daykin and Erd\H os~\cite{BDE} lowered the value of $n_0(r,t)$ to $2tr^3$. More recently, Huang, Loh, and Sudakov~\cite{HLS} lowered it further to $3tr^2$, which was further improved slightly in~\cite{FLM} (see also ~\cite{LM}).  Frankl and F\"uredi (unpublished) gave an improvement for large $r$ by showing that $n_0(r,t)< c r t^2$ for some positive $c$.  Finally, Frankl~\cite{Frankl2013} gave the following improvement that contains all of the previous results on this problem.   Let
$${\cal A}(n,r,t):=\left\{e \in {[n] \choose r}: e \cap [t] \ne \emptyset\right\}.$$

\begin{theorem} {\bf (\cite{Frankl2013})}\label{franklmatching}
Let $r, t \ge 1$ and $n \ge (2t+1)r-t$. Then
$$\ex_r(n, M_{t+1}^+)= {n \choose r}- {n-t \choose r}.$$
Equality holds only for families isomorphic to ${\cal A}(n,r,t)$.
\end{theorem}

\subsection{Stars}
For $t \ge 2$, let $S_t=K_{1,t}$ be the star with $t$ edges. The problem of determining $\ex_r(n, S_t^+)$ was first considered in 1976 by Erd\H os and S\'os~\cite{S}. They determined  $\ex_3(n, S_2^+)$ precisely for all $n$.

\begin{theorem} {\bf (\cite{S})} \label{erdossos}
$$\ex_3(n, S_2^+) = \begin{cases}
n & \hbox{ if $n  \equiv 0$ (mod 4)} \\
n-1 & \hbox{ if $n  \equiv 1$ (mod 4)} \\
n-2 & \hbox{ $if n  \equiv 2,3$ (mod 4).} \\
\end{cases}
$$
\end{theorem}
The extremal families above are basically the disjoint unions of triple systems consisting of all 3-subsets of a 4-subset (in the last case there are some more  possibilities, namely to take some 3-sets containing two fixed points). The
case $r \ge 4$ is more complicated.
Erd\H os and S\'os conjectured that
\begin{equation} \label{esf} \ex_r(n, S_2^+)={n-2 \choose r-2}\qquad \hbox{ for} \quad  n>n_0(r)
\end{equation}
and
this was proved soon after by Frankl~\cite{Frankl1977} (the lower bound is trivial for all $n \ge r\ge 4$). Determining the correct value of $n_0(r)$ above seems very difficult.  The following bound
proved by Keevash-Mubayi-Wilson~\cite{KMW} is slightly sub-optimal,  but holds for all $n\ge r\ge 4$:
$$\ex_r(n, S_2^+) \leq {n \choose r-2}.$$

 For fixed $r,t\ge 2$, Duke and Erd\H os~\cite{DE} showed that $\ex_r(n, S_t^+)=\Theta(n^{r-2})$ and conjectured that
 $\ex_3(n, S_3^+)=6(n-3)+2$ (presumably for $n \ge 7$ though this is not explicitly mentioned).  Frankl~\cite{Frankl1978}   proved that
 $$\ex_3(n, S_3^+)=6(n-3)+2\qquad \hbox{  for} \quad n \ge 54$$
and $\ex_3(n, S_t^+)< (5/3)t(t-1)n$. Chung~\cite{C}
substantially improved these bounds  and finally, Frankl and Chung~\cite{FC} improved them further as follows.

\begin{theorem} {\bf (\cite{FC})} \label{franklchung}
$$\ex_3(n, S_t^+) = \begin{cases}
nt(t-1)+2{t \choose 3} & \hbox{ if $n>t(t-1)(5t+2)/2$ and $t\ge 3$ is odd} \\
\frac{nt(2t-3)}{2}-\frac{2t^3-9t+6}{2} & \hbox{ if
$n\ge 2t^3-9t+7$ and $t \ge 4$ is even.} \\
\end{cases}
$$
\end{theorem}
Frankl and Chung characterized the extremal example as well.  The only other case where exact results are known for all $n$ is $r=4$ and $t=2$.
Before stating this result
we should recall that full version of the
 Erd\H{o}s-Ko-Rado \cite{EKR} theorem. Say
that an $r$-graph $H$ is {\em $t$-intersecting}
if for every $e,f \in H$ we have $|e \cap f| \geq t$.
They showed that, if $H$ is a $t$-intersecting $r$-graph on $[n]$
with $n$ sufficiently large, then
$|H| \leq \binom{n-t}{r-t}$. (The case $t=2$ is
pertinent to our current discussion.) Confirming a
conjecture of Erd\H{o}s, Wilson \cite{W} showed that this
bound in fact holds for $n \geq (t+1)(r-t+1)$ (which is
the best possible strengthening), and furthermore that the
unique maximum system consists of all $r$-sets containing
some fixed $t$-set. To describe the complete solution
for all $n$ we need to define the $t$-intersecting
systems
$$\mathcal{F}(n,r,t,i) = \left\{e \in {[n] \choose r} : |e \cap [t+2i]| \geq t + i \right\}$$ for $0 \leq i \leq r-t$.
The complete intersection theorem, conjectured by Frankl, and
proved by Ahlswede and Khachatrian \cite{AK}, is
that a maximum size  $t$-intersecting $r$-graph
on $[n]$ is isomorphic to $\mathcal{F}(n,r,t,i)$
for some $i$ which can easily be computed given $n$. Note
that $\mathcal{F}(n,r,t,0)$ is isomorphic to the $r$-graph of all $r$-sets containing
some fixed $t$-set.

\begin{theorem} {\bf (\cite{KMW})}\label{kmw}
$$\ex_4(n, S_2^+) \leq \left\{ \begin{array}{cc}
\binom{n}{4} & n=4,5,6 \\ 15 & n=7 \\ 17 & n=8 \\ \binom{n-2}{2}
& n \geq 9 \end{array} \right.$$
Furthermore, the only cases of equality are ${[n] \choose 4}$ for
$n=4,5$, $\mathcal{F}(n,4,2,2) = {[6] \choose 4}$ for $n=6,7$,
$\ \mathcal{F}(8,4,2,1)$ for $n=8$ and
$\mathcal{F}(n,4,2,0)$ for $n \geq 9$.
\end{theorem}

The following stability version of the above result was also shown in~\cite{KMW}

\begin{theorem} \label{kmwstability}
For any $\epsilon>0$ there is $\delta>0$ such that if
$H$ is an $r$-graph on
$[n]$ with no singleton intersection and
$|H| \geq (1-\delta)\binom{n-2}{r-2}$, then
there are two points $x,y$ so that all but at most
$\epsilon n^{r-2}$ sets of $H$ contain both $x$ and $y$.
\end{theorem}

\subsection{Paths and Cycles}
In this section we consider the case when $G=P_t$ or $G=C_t$, where $P_t$ ($C_t$) is the path (cycle) with $t$ edges.   Erd\H{o}s and Gallai~\cite{EG} proved that $\ex_2(n,P_t) \leq \frac{t-1}{2}n$ and this is tight whenever $t|n$.  Since $P_2=S_2$, we have already seen in (\ref{esf}) that
Frankl~\cite{Frankl1977}  determined $\ex_r(n,P_2^+)$ for large $n$.  F\"uredi, Jiang and Seiver~\cite{FJS} determined
$\ex_r(n, P_t^+)$ precisely for all $r \ge 4$, $t \ge 3$ and $n$ large while also characterizing the extremal examples. This improved results of the authors~\cite{MV1}
and also settled a conjecture from~\cite{MV1}.
They conjectured a similar result for $r=3$.

The case  $G=K_3=C_3$ is also well-researched~\cite{CK,FF, MV}, indeed, when $r=2$ this
 is precisely Mantel's theorem from 1907.  Frankl and F\"{u}redi~\cite{FF} showed that the unique  extremal $r$-graph
 on $[n]$ not containing $C_3^+$ is a star, for large enough $n$.
 For $r=3$,  Cs\'ak\'any  and Kahn~\cite{CK} proved the same result for all $n \ge 6$.
 More recently, F\"uredi and Jiang~\cite{FJ} determined the extremal function for $C_t^+$ for all $t \ge 3$,
 $r \ge 5$ and large $n$; their results substantially extend earlier results of Erd\H os and settled a conjecture
 of the current authors~\cite{MV1} for $r \ge 5$. Recently, Kostochka and the current authors~\cite{KMV} further extended the results in~\cite{FJ} to the case of $r=3,4$.
To state the result for all $r,t$, we need some notation.
For $L\subset [n]$ let $S_{L}^r(n)$
denote the $r$-graph on $[n]$ consisting of all $r$-element subsets of $[n]$
intersecting $L$.

\begin{theorem} {\bf (\cite{FJS, FJ}, \cite{KMV})}  \label{main-path}Let $r \geq 3$, $t\geq 4$, and $\ell = \lfloor \frac{t-1}{2} \rfloor$. For sufficiently large $n$, 
\[
\ex_r(n,P_t^+) = {n \choose r} - {n - \ell \choose r} + \left\{\begin{array}{ll}
0 & \mbox{ if }t\mbox{ is odd}\\
{n - \ell - 2 \choose r - 2} & \mbox{ if }t\mbox{ is even}
\end{array}\right.
\]
with equality only for $S_L^r(n)$ if $t$ is odd and $S_L^r(n) \cup F$ where $F$ is extremal for $\{P_2^+,2P_1^+\}$ on $n - \ell$ vertices if $t$ is even.
The same result holds for $C_t^+$ except  the case $(t,r) = (4,3)$, in which case
\[ \ex_3(n,C_4^+) = {n \choose r} - {n - 1 \choose r} + \max\{n-3,4\lfloor \tfrac{n-1}{4}\rfloor\}\]
with equality only for triple systems of the form $S_L^3(n) \cup F$ where $F$ is extremal for $P_2^+$ on $n - 1$ vertices.
  \end{theorem}

 It was recently shown by Bushaw and Kettle~\cite{BK} that the Tur\' an problem for
disjoint $t$-paths can be easily solved once we know the extremal function for a single $t$-path.
As Theorem \ref{main-path} solves the $t$-paths problem for all $r \ge 3$, the corresponding extremal questions for disjoint $t$-paths are
also completely solved (for large $n$).
A similar situation likely holds for disjoint $t$-cycles, as recently observed by Gu, Li and Shi~\cite{GLS}.

Finally, we remark that the results of~\cite{FJS, FJ, KMV} also give stability versions of Theorem~\ref{main-path}.  Here is a sample result.

\begin{theorem} {\bf (\cite{KMV})} Fix $r \ge 3, t\ge 4$ $\ell = \lfloor \frac{t-1}{2} \rfloor$ and let $H$ be an $n$-vertex $r$-graph with $|H|\sim \ell{n \choose r-1}$ containing no $P_t^+$ or no $C_t^+$.  Then there exists  $H' \subset \partial_{r-1}H$ with $|H'|\sim {n \choose r-1}$ and a set $L$ of $\ell$ vertices of $H$ such that $N_{H'}(e)=L$ for every $e \in H'$. In particular, $|H- L|=o(n^{r-1})$.
\end{theorem}

\subsection{Trees}

In this section we consider the case when $G$ is acyclic. What makes this problem more complicated than the previous ones where $G$ is a star or path, is that we can produce constructions that exploit the structure of $G$. Let us make this precise.
 A set of vertices in a hypergraph $F$ containing exactly one vertex from every edge of  $F$ is called a {\em crosscut},
following Frankl and F\"{u}redi~\cite{FF}. Let $\sigma(F)$ be the minimum size of a crosscut of $F$ if it exists, i.e.,
\[ \sigma(F) := \min\{|X| : X \subset V(F), \forall e \in F, |e \cap X| = 1\}\]
 if such an $X$ exists.

Since the $r$-graph on $n$ vertices consisting of all edges containing exactly one vertex from a fixed subset of size $\sigma(F) - 1$
does not contain $F$, we have
\begin{equation} \label{triviallower}\ex_r(n,F) \geq (\sigma(F) - 1){n - \sigma(F) + 1 \choose r - 1} \sim (\sigma(F) - 1){n \choose r - 1}.\end{equation}
An intriguing open question is when asymptotic equality holds above and this is one of the main open problems in this area. Indeed, it appears that the parameter $\sigma(F)$ often plays a crucial role in determining the extremal function for $F$.  F\"{u}redi~\cite{Furedi} determined the asymptotics of $\ex_r(n, G^+)$  when $G$ is a forest and $r \geq 4$ and conjectured a similar result for $r=3$.  This conjecture was recently proved by Kostochka and the current authors~\cite{KMV3}.

\begin{theorem} {\bf (\cite{Furedi} for $r\ge 4$, \cite{KMV3} for $r=3$)} \label{main-tree}
Fix $r \ge 3$  and a forest $G$. Then
\[ \ex_r(n,G^+) \sim (\sigma(G^+) - 1){n \choose r-1}.\]
\end{theorem}

Stability and exact results for Theorem~\ref{main-tree} do not follow using current proof methods.  It would be interesting to obtain such results.

\subsection{Graphs with crosscuts of size two}

If $\sigma(G^+)=1$, then clearly $G=S_t$ and this case has already been discussed.  Here we focus on the next case, $\sigma(G^+)=2$.  It is straightforward to see that any graph $G$ with $\sigma(G^+) = 2$ is contained in a star with an edge added, or a complete bipartite graph with one part of size two. The following result of
Kostochka and the current authors~\cite{KMV2} gives an asymptotically exact result in these cases:

\begin{theorem} {\bf (\cite{KMV2})}\label{k2t-expansion}
For every fixed graph $G$ with $\sigma(G^+) = 2$,
 $$\ex_3(n,G^+) \sim {n \choose 2}.$$
 In particular, for any graph $G$, $\ex_3(n,G^+) \leq \left(\tfrac{1}{2} + o(1)\right)n^2$ or $\ex_3(n,G^+) \geq (1 + o(1))n^2.$
\end{theorem}

The last statement in Theorem \ref{k2t-expansion} follows easily from the first, since if $\sigma(G^+) = 3$ then
$\ex_3(n,G^+) \geq (1 + o(1))n^2$. Theorem \ref{k2t-expansion} determines the asymptotic behavior of $\ex_3(n,G^+)$ when $\sigma(G^+) = 2$;
the next case $\sigma(G^+) = 3$ is an avenue of further research.
It does appear, however, that determining the order of magnitude of $\ex_3(n,G^+)$ is more difficult when $\sigma(G^+) = 3$;
it is shown in~\cite{KMV3} that if $G = K_{3,t}$ then $\ex_3(n,G^+)/n^2 \rightarrow \infty$ as $t \rightarrow \infty$.

\subsection{Complete bipartite graphs}

The general problem of determining $\mbox{ex}_3(n,G^+)$ when $G$ is bipartite was studied in~\cite{KMV3}.
In this case, the problem is related to maximizing the number of triangles in a $G$-free graph: if one takes as the hypergraph
the triangles in the graph, then one obtains a triple system not containing $G^+$. This framework was studied by
Alon and Shikhelman~\cite{AS} in the context of $K_{s,t}$-free graphs who considered the following question: what is the maximum number of triangles in a
$K_{s,t}$-free graph on $n$ vertices? Kostochka and the current authors~\cite{KMV3} give a short proof of one of their results, and in particular, prove the following more general theorem.

\begin{theorem} {\bf (\cite{KMV3})} \label{kst-expansion}
For $t \geq s \geq 3$ fixed, $\ex_3(n,K_{s,t}^+) = O(n^{3 - \frac{1}{s}})$ and $\ex_3(n,K_{s,t}^+) = \Theta(n^{3 - \frac{1}{s}})$ for $t > (s - 1)!$.
\end{theorem}

We give the proof of this theorem in Section \ref{proof-kst}. Theorem \ref{kst-expansion} does not extend to the case $s = 2$,
and in that case $\ex_3(n,K_{2,t}^+) \sim {n \choose 2}$ as shown by Theorem \ref{k2t-expansion}.  For a general bipartite graph $G$,
the following is proved in~\cite{KMV3}.

\begin{theorem}\label{generalbipartite}
Let $G$ be a bipartite graph such that $\ex(n,G) = O(n^{1 + \alpha})$ and $\ex(n,G - \{v\}) = O(n^{1 + \beta})$ for
every $v \in V(G)$. Then $\ex_3(n,G^+) = O(n^{1 + \alpha + \alpha \beta}) + O(n^2)$.
\end{theorem}

Let $Q$ denote the 3-dimensional cube graph.
Erd\H{o}s and Simonovits~\cite{ErdosSimCube} showed $\mbox{ex}(n,Q) = O(n^{8/5})$ and $\mbox{ex}(n,Q - \{v\}) = O(n^{3/2})$ for all $v \in Q$.
Applying Theorem \ref{generalbipartite}, we find $\mbox{ex}_3(n,Q^+) = O(n^{1.9}) + O(n^2) = O(n^2)$ and therefore $\mbox{ex}_3(n,Q^+) = \Theta(n^2)$.
In fact for any bipartite graph $G$ (except a star) such that $\mbox{ex}(n,G) = O(n^{1 + \alpha})$ where $1 + \alpha$ is less than the golden ratio
$(1 + \sqrt{5})/2$, one has $\mbox{ex}_3(n,G^+) = \Theta(n^2)$. This includes the family of all bipartite graph in which the
vertices in one part have degree two, since for any such graph $G$, $\ex_2(n,G) = O(n^{3/2})$ using the method of dependent random
choice~\cite{FoxS}:

\begin{corollary}
Let $G$ be a bipartite graph such that all the vertices in one part have degree at most two.
Then $\ex_2(n,G^+) = O(n^2)$.
\end{corollary}

\subsection{General expansions}

If $G$ is a 3-colorable graph, then $\chi(G^+) \le 3$ and therefore $\ex_3(n,G^+) = o(n^3)$. In the last section,
we saw examples of bipartite graphs for which the order of magnitude of $\ex_3(n,G^+)$ can be determined. One may ask for which
3-colorable graphs $G$ one has $\ex_3(n,G^+) = O(n^2)$. Theorem \ref{main-path} shows that for any odd cycle $C$, $\ex_3(n,C^+) \sim (\sigma(C) - 1){n \choose 2}$. The following result gives a superquadratic lower bound in many cases.

\begin{theorem} {\bf (\cite{KMV3})}\label{acyclic}
Let $G$ be a graph such that in every proper coloring of $G$, every pair of color classes induces a graph containing a cycle.
Then for some $c > 2$, $\ex_3(n,G^+) = \Omega(n^{c})$.
\end{theorem}

A construction proving Theorem \ref{acyclic} is simple. Suppose $|V(G)| = k$. Consider the 3-partite triple system $H$ with
parts $U,V$ and $W$ where $|U| \leq |V| \leq |W| \leq |U| + 1$, and let $J$ be an extremal bipartite graph
containing no cycle of length less than $k + 1$ and with parts $U$ and $V$. Then the edges
of $H$ consist of all triples $\{u,v,w\}$ such that $\{u,v\} \in J$ and $w \in W$. The number
of edges in $H$ is $|J||W|$. It is well-known that an extremal bipartite graph with $n$ vertices and no cycles of length up to $k$
has $\Omega(n^{c_k})$ edges, for some $c_k > 1$, so if $H$ has $n$ vertices, then $|H| = |J||W| = \Omega(n \cdot n^{c_k})$.
This establishes Theorem \ref{acyclic} with $c = 1 + c_k$.

\medskip

A second construction using random graphs works when $G$ is a graph with many edges. Specifically,
let $G$ be a graph containing a cycle and let $m(G) = \max\{\frac{|H| - 1}{|V(H)| - 2} : H \subseteq G, |V(H)| \geq 3\}$.
It is known that there exists $\delta > 0$ such that with positive probability, the random graph
$G_{n,p}$ has a $G$-free subgraph with at least $\delta p^3 n^3$ triangles when $p = \delta n^{-1/m}$.
If $H$ is the triple system comprising the triangles in the subgraph of $G_{n,p}$, then $H$ does not
contain $G^+$. This gives the following theorem:

\begin{theorem}  \label{dense}
For any graph $G$, with $m=m(G)$,
\[ \ex_3(n,G^+) = \Omega(n^{3 - \frac{3}{m}}).\]
In particular, for any graph $G$ with more than $3|V(G)| - 5$ edges, there exists $c > 2$ such that
\[ \ex_3(n,G^+) = \Omega(n^c).\]
\end{theorem}

An {\em acyclic coloring} of a graph is a proper vertex-coloring of the graph such that the union of any two color
classes induces a forest. If $G$ is a graph with an acyclic 3-coloring, then $|G| \leq 2|V(G)| - 3$ and
Theorems \ref{dense} and \ref{acyclic} do not apply. It seems natural therefore to ask
whether $\ex_3(n,G^+) = O(n^2)$ when $G$ has an acyclic 3-coloring.

\section{Proof Methods}

In this section, we describe some of the important techniques  that have been successful in proving results about extremal numbers of expansions.

\subsection{Full subgraphs}\label{fullsubs}

It is a well-known fact that any graph $G$ has a subgraph of minimum degree at least $d+1$ with at least $|G| - d|V(G)|$ edges.
In this section, we state the simple  extension of  this fact to hypergraphs.

\begin{definition}
For $r > s \geq 1$ and $d \geq 1$, an $r$-graph $H$ is {\em $(d,s)$-full} if
$d(e) \geq d$ for all $e \in \partial_{s}H$.
\end{definition}

Thus $H$ is $(d,s)$-full is equivalent to saying that the minimum non-zero degree of $(s - 1)$-sets in $V(H)$ is at least $d$.  When $s=2$ we will use the simpler notation $d$-full.

\begin{lemma} \label{fullsub}
For $d \geq 1$ and $r > s \geq 1$, every $n$-vertex $r$-graph $H$ has a $(d + 1,s)$-full subgraph $F$ with
\[ |F| \geq |H| - d|\partial_s H|.\]
\end{lemma}

\begin{proof}
A {\em $d$-sparse sequence} is a maximal sequence $e_1,e_2,\dots,e_m \in \partial_s H$ such that $d_H(e_1) \leq d$, and for all $i > 1$, $e_i$
 is contained in at most $d$ edges of $H$ which contain none of $e_1,e_2,\dots,e_{i - 1}$.
The $r$-graph $F$ obtained by deleting all edges of $H$ containing at least one of the $e_i$ is $(d + 1)$-full. Since
a $d$-sparse sequence has length at most $|\partial_s H|$, we have $|F| \geq |H| - d|\partial_s H|$.
\end{proof}

Although Lemma \ref{fullsub} and its proof is simple, it is a very important tool and is used frequently in the proofs of the theorems in this survey.

\subsection{K\"{o}vari-S\'{o}s-Tur\'{a}n Theorem}

The problem of determining the order of magnitude of $\ex_2(n,G)$ is generally notoriously difficult when $G$ is a
bipartite graph. Even for such small graphs as $G = C_8$, $G = Q_3$ and $G = K_{4,4}$, this order of magnitude is not
known. Zarankiewicz~\cite{Z} conjectured that $\ex_2(n,K_{s,t}) = \Theta(n^{2 - 1/s})$ whenever $t \geq s \geq 2$,
and this conjecture remains open in general. K\"{o}vari, S\'{o}s and Tur\'{a}n~\cite{KST} gave the following general upper bound
for $\ex_2(n,K_{s,t})$.

\begin{theorem}\label{kst}
For all $t \geq s \geq 2$,
\[ \ex_2(n,K_{s,t}) \leq \tfrac{1}{2}[(t - 1)^{1/s} n^{2 - 1/s} + (s - 1)n].\]
\end{theorem}

When $t > (s - 1)! \geq 1$, the norm graph constructions
of Alon, R\'{o}nyai and Szab\'{o}~\cite{ARS} show that Theorem \ref{kst-expansion} is tight up to constants, and so this verifies the
Zarankiewicz~\cite{Z} conjecture in those cases. In addition, F\"{u}redi~\cite{F-k2t} determined the asymptotic behavior of
$\ex_2(n,K_{2,t})$ and also $\ex_2(n,K_{3,3})$ -- a construction of Brown~\cite{Brown} gives the tightness of F\"{u}redi's upper
bound for $\ex_2(n,K_{3,3})$.

The short proof of Theorem \ref{kst-expansion} exclusively uses Theorem \ref{kst}  in combination with Lemma \ref{fullsub}, and we will present this proof in Section \ref{proof-kst}.

\subsection{Delta-systems}

In this section we describe the Delta-system method initiated by Deza, Erd\H os and Frankl~\cite{DEF}.
The proof of Theorem \ref{main-tree} for $r \geq 4$ due to F\"uredi~\cite{Furedi} uses this approach. That proof does not extend to $r = 3$,  and different techniques are needed for this case (these will be discussed in detail in sections 4.4 and 4.5).
We begin with the basic definitions.

\medskip

A {\em Delta-system} is a hypergraph $\triangle$ such that for any distinct
edges $e,f \in \triangle$, we have $e \cap f = \bigcap_{g \in \triangle} g$.
We denote by $\triangle_{r,s}$ the $r$-uniform Delta-system with $s$ edges and
the {\em core} of a Delta-system $\triangle$ is $\mbox{core}(\triangle) = \bigcap_{g \in \triangle}g$.
If $H$ is any $r$-graph and $f \subset V(H)$, then the {\em core degree of $f$}
is
\[ \d_H(f) = \max\{s : \exists  \triangle_{r,s} \subset H, \mbox{core}(\triangle_{r,s}) = f\}.\]
The following basic lemma is observed in~\cite{FJS}:

\begin{lemma} \label{kernelembed}
Let $r \ge 3$, $H$ be an $r$-graph and $G$ be a graph with $V(G) \subset V(H)$. Suppose that
 $\d_H(e) \geq r|G|$ for all $e \in G$. Then $G^+ \subset H$.
\end{lemma}

\begin{proof}
Let $F = \{e_1,e_2,\dots,e_k\}\subset {V(H) \choose 2}$ with $F \cong G$ and $F_i = \{e_1, \ldots, e_i\}$ for $i \in [k]$. We will show by induction on $i$ that $F_i^+ \subset H$. The case $i=0$ is trivial, as we may assume that $F_0=\emptyset$. For the induction step, assume that $F_i^+ \subset H$,
so we have edges $f_1,f_2,\dots,f_i \in H$
such that $e_j \subset f_j$ for $j \le i$ and the sets $f_j \backslash e_j$ are pairwise disjoint.
Since $\d_H(e_{i+1}) \geq rk$,
$e_{i+1}$ is contained in edges $g_1,g_2,\dots,g_{rk} \in H$ forming a Delta-system $\triangle$ with $\mbox{core}(\triangle) = e_{i+1}$.
Since $\sum_{j = 1}^{i}|f_i| = r(k - 1)$, and the sets $g_j \backslash e_{i+1}$ are pairwise disjoint,
there exists $f_{i+1} \in \{g_1,g_2,\dots,g_{rk}\}$ such that $f_{i+1} \cap f_j = \emptyset$ for all $j \le i$.
Then $F_{i+1}^+ = F_i^+ \cup \{f_{i+1}\} \subset H$.
\end{proof}

\subsubsection{Intersection semilattice lemma}

A basic result about finite sets is the Erd\H{o}s-Rado Sunflower Lemma~\cite{ER}:

\begin{lemma} {\bf (Erd\H{o}s-Rado Sunflower Lemma~\cite{ER})} \label{sunlemma}
If $F$ is a collection of sets of size at most $k$ and $|F|\geq k!(s-1)^{k}$,
then $F$ contains a Delta-system with $s$ sets.
\end{lemma}

The driving force behind the Delta-system method is the {\em intersection semilattice lemma} of
F\"uredi~\cite{Furedidelta}, which is a far reaching generalization of the the Erd\H os-Rado sunflower lemma.
If $H$ is a hypergraph and $e \subseteq V(H)$, define
\[ H|_e = \{e \cap f : f \in H, e \neq f\}.\]
If $H$ is an $r$-graph with $r$-partition $(X_1,X_2,\dots,X_r)$,  define the projection of $e$ to be $\proj(e)=\{i: e \cap X_i \ne \emptyset\}$.
Then the {\em intersection pattern of $e$} is the
hypergraph
\[ I_H(e) = \{ \proj(g) : g \in H|_e\} \subset 2^{[r]}.\]

\begin{lemma} \label{semilattice} {\bf (The intersection semilattice lemma)}
For any positive integers $r,s$ there exists a positive real $c:=c(r,s)$ such that
if $H$ is an $r$-graph, then there is an $r$-partite $r$-graph $H^* \subset H$ with
parts $(X_1,X_2,\dots,X_r)$ and $|H^*| > c|H|$ and a hypergraph $J \subset 2^{[r]}$ such that
\begin{center}
\begin{tabular}{lp{5in}}
{\rm (1)} & $J$ is intersection-closed. \\
{\rm (2)} & For all $e \in H^*$, $I_{H^*}(e) = J$. \\
{\rm (3)} & For all $e \in H^*$ and $g \in H^*|_e$, $\d_{H^*}(g) \geq s$.
\end{tabular}
\end{center}
\end{lemma}

An $r$-partite $r$-graph $H^*$ satisfying (1) -- (3) is referred to as {\em $(s,J)$-homogeneous}.
An $r$-partite $r$-graph is {\em $J$-homogeneous} if it is $(s,J)$-homogeneous for some positive integer $s$.
For example, the complete $r$-partite $r$-graph with parts of size $t \geq 2$ is $(t,J)$-homogeneous with
$J = 2^{[r]}$.

\subsubsection{Homogeneous hypergraphs and rank}

The goal of this section is to prove Lemma~\ref{deltaconsequence} below which produces a partition of an $r$-graph which does not contain the expansion of a forest into two parts. The first part has few edges, while the second part has a rich structure. The machinery used to prove this decomposition lemma heavily employs Lemma~\ref{semilattice} and the entire approach constitutes the Delta-system method.

The {\em rank} of a hypergraph $J \subset 2^{[r]}$ is
\[ \rho(J) = \min\{|e| : e \subset V(J), d_J(e) = 0\}.\]
For $n$-vertex $J$-homogeneous $r$-graphs $H^*$ with $|H^*| > {n \choose k}$ in Lemma~\ref{semilattice},
we see that $\rho(J) > k$.

\begin{lemma} \label{rank} Let $r$ be a positive integer and $H^*$ an $n$-vertex $J$-homogeneous $r$-graph,
such that $\rho(J) = k$. Then $|H^*| \le {n \choose k}$.
\end{lemma}

\begin{proof}
Since $\rho(J) = k$, every edge $e \in H^*$ contains a $k$-set $S_e$ such that
for every $f \in H^* \backslash \{e\}$, $S_e \not \subset f \cap e$. Then the map $\phi : H^* \rightarrow {[n] \choose k}$
defined by $\phi(e) = S_e$ is injective, so $|H^*| \leq {n \choose k}$.
\end{proof}

Let us now prove a very simple consequence of the semilattice lemma.

\begin{corollary}\label{easy}
Let $r > s$, $q \geq 1$, let $G$ be a $q$-edge forest, and let $H^*$ be a non-empty $(rq,J)$-homogeneous $r$-graph such that $\rho(J) = r$.
Then $G^+ \subset H^*$.
\end{corollary}

\begin{proof}
Since $\rho(J) = r$, every edge $e \in H^*$ has the property that whenever $x \in e$,
$\d_{H^*}(e \backslash \{x\}) \geq rq$. Since $J$ is intersection-closed,
$\d_{H^*}(f) \geq rq$ for every non-empty $g \subset V(J)$ with $|g| \leq r - 1$ and $\emptyset\ne f \subset \prod_{i \in g} X_i$. Therefore every edge $f$ of $\partial H^*$ has $\d_{H^*}(f) \geq rq$, and
$\partial H^*$ is a graph of minimum degree at least $rq$. In particular,
$G \subset \partial H^*$. Applying Lemma \ref{kernelembed}
completes the proof.
\end{proof}

The heart of the matter for embedding expansions in hypergraphs with
roughly ${n \choose r - 1}$ edges is therefore the case $\rho(J) = r - 1$ in the semilattice lemma.
If $\rho(J) = r - 1$ in the semilattice lemma, then every $(r - 2)$-set in $V(J)$ is contained in some edge of $J$ (but is not necessarily an edge of $J$). The following lemma shows why $r \ge 4$ is needed to apply the Delta-system method to the current problem.
A hypergraph $J \subset 2^{[r]}$ is {\em central} if there exists $x \in [r]$ such that $J$ contains all $(r - 1)$-sets containing $x$ but not $[r] \backslash \{x\}$.

\begin{lemma} \label{r>3}
Let $J \subset 2^{[r]}\setminus[r]$ be an intersection-closed hypergraph with $\rho(J) = r - 1$ such that $J$ is not
central. If $r \geq 4$, then $J$ contains at least two singleton edges.
\end{lemma}

\begin{proof}
Let $e_i = [r] \backslash \{i\}$. If $e_i \in J$ for $1 \leq i \leq r$, then since $J$ is intersection closed, every element of $[r]$ is
a singleton edge in $J$. Suppose $e_1,e_2,\dots,e_t \not \in J$ and $e_{t + 1},e_{t + 2},\dots,e_r \in J$. Since $J$ is not
central, $t \geq 2$. Since $\rho(J) = r - 1$ and $|e_i \cap e_j| = r - 2$, $e_i \cap e_j \in J$ for $1 \leq i < j \leq t$.
Since $J$ is intersection-closed, $e_i \cap e_j \in J$ for $t < i < j \leq r$. Since $J$ is intersection-closed,
$\{i\} \in J$ for each $i \in [t]$ if $t \ge 3$, and $\{i\} \in J$ for $3 \le i \le r$ if $t = 2$. Since $r \geq 4$,
these are the required singletons.
\end{proof}

The purpose of the next lemma is to show that if $H^*$ is an $r$-graph that is $(s,J)$-homogeneous
and $J$ is not central, then we can expand any forest in $H^*$ unless $|H^*| = o(n^{r - 1})$.

\begin{lemma} \label{embed}
Let $q \geq 1$, $r \geq 4$. Let $G$ be a $q$-edge tree and $H^*$ an $(rq,J)$-homogeneous $r$-graph with $\rho(J) \geq r-1$ and $J$ not central. Then $G^+ \subset H^*$.
\end{lemma}

 \begin{proof}
If $\rho(J) = r$, then apply Corollary \ref{easy}.  For $\rho(J) = r - 1$, the proof is a standard greedy embedding algorithm,
and we proceed by induction on $q$, the case $q = 1$ being trivial.
By Lemma \ref{r>3} and (2) and (3) in Lemma \ref{semilattice}, for each $e \in H^*$ there exist vertices $u_e,v_e \in e$ such that
$\d_{H^*}(u_e) \geq rq$ and $\d_{H^*}(v_e) \geq rq$. Let $C = \{u_e,v_e : e \in H^*\}$.
The key observation is that in order to do the induction step, we will map $V(G)$ to $C$.
For $q \ge 2$, let $v$ be a leaf of $G$, $u$ is its unique neighbor in $G$ and $G_1 = G - \{v\}$.
By induction, there exists a graph isomorphism $\phi : G_1 \rightarrow L$ such that $L \subset \partial H^*$
and $V(L) \subset C$ and $L^+ \subset H^*$. Then $\d_{H^*}(\phi(u)) \geq rq$, so there is a Delta-system $\triangle = \triangle_{r,rq} \subset H^*$ with $\mbox{core}(\triangle) = \{\phi(u)\}$. Therefore there is an edge $f \in \triangle$ such that $f \setminus\{\phi(u)\}$ is disjoint from $V(L)$. Moreover, by
Lemma~\ref{r>3} we can find another vertex $w \in f$ with $\d_{H^*}(w) \geq rq$. Extending $\phi$ by defining $\phi(v) = w$, we
have an isomorphism $\phi : G \rightarrow L \cup \{v,w\}$ such that $L \subset \partial H^*$ and
$L^+ \cup \{f\}$ is an embedding of $G^+$ in $H^*$.
\end{proof}

The intersection semilattice lemma was tailored to the problem of embedding expansions
of forests in $r$-graphs for $r \geq 4$ by F\"uredi~\cite{Furedi,FJS}.  The following lemma is the main consequence of the Delta-system method that is used to embed expansions of forests in $r$-graphs.  We will use it in Section 5.2 to give a proof of Theorem~\ref{main-tree} for $r \ge 4$.

\begin{lemma} \label{deltaconsequence}
For $v \geq 1$ and $r \geq 4$, there exists a constant $c = c(v,r)$ such that if $H$ is an $n$-vertex $r$-graph and $G$ a $v$-vertex forest,
 and $G^+ \not \subset H$, then there is a partition $H = H_1 \sqcup H_2$ of $H$ such that $|H_1| \leq  c{n - 2 \choose r - 2}$ and
 for all $f \in H_2$, there exists $x \in f$ such that $\d_{H_2}(\{x,y\}) \geq rv$ for all $y \in f \backslash \{x\}$.
\end{lemma}

\begin{proof}
Apply Lemma~\ref{semilattice} to $H$ to obtain an $(r,s)$-homogeneous family $F_1$
with intersection pattern $J_1$ and $|F_1|> c'|H|$ (where $c':=c(r,s)$). Now apply Lemma~\ref{semilattice} again to $H-F_1$ to obtain $F_2$ and $J_2$, where $|F_2|>c'(|H|-|F_1|)$. Continue this process and let $m$ be the smallest nonnegative integer such that $\rho(J_{m+1})\le r-2$.  Put $H_2=\bigcup_{i=1}^m F_i$ and $H_1=H-H_2$. Then $|F_{m+1}| > c'|H_1|$ and, by Lemma~\ref{rank}, we have $|F_{m+1}| \le {n \choose r-2}$. Setting $c = 2/c'$ we have $|H_1|\le c {n \choose r-2}$ as desired. An edge not in $H_1$ is in some $F_i$ for $i \le m$, and $\rho(J_i) \ge r-1$.  If $r(J_i)=r-1$ or $r(J_i)=r-2$ and $J_i$ is not central, then $G^+ \subset H$ by Lemma~\ref{embed}, so if $f \in F_i\setminus H_1$,
then $J_i$ is central with rank $r-1$. Let $x$ be a  vertex of $[r]$ such that $J_i$ contains all $(r-1)$-sets containing $x$ but not $[r]\setminus\{x\}$. Since $J_i$ is intersection closed, all proper subsets of $f$ containing $x$ are cores of some Delta-system $\Delta_{r,s}$ in $F_i \subset H_2$, including all the subsets of the form $\{x,y\}$ where $y \in f\setminus\{x\}$.
Consequently, $\d_{H_2}(\{x,y\}) \geq rv$ for all $y \in f \backslash \{x\}$.\end{proof}

\subsection{Ramsey theory, shadows  and random sampling}\label{scanonical ramsey}

In this section we introduce another method for proving results on expansions. This approach is more recent, and was developed by Kostochka and the current authors in~\cite{KMV, KMV2, KMV3}
to prove Theorems~\ref{main-path}, \ref{main-tree} and \ref{k2t-expansion}. The method applies to $r$ graphs for $r \ge 3$, but here we restrict attention to the case of triple systems.

\subsubsection{Bipartite canonical Ramsey theorem}\label{canonical ramsey}

One of the main new ideas is to use the canonical Ramsey theorem of Erd\H{o}s and Rado~\cite{ER}.
We need a bipartite version of this classical result, for which the following definitions are used.

\begin{definition} Let $F$ be a bipartite graph with parts $X$ and $Y$ and an edge-coloring
$\chi$. Then
\begin{center}
\begin{tabular}{lp{5in}}
$\bullet$ &  $\chi$ is {\em $Z$-canonical} for $Z \in \{X, Y\}$ if for each $z \in Z$, all edges of $F$ on  $z$ have the same color and edges on different vertices in $Z$ have different colors \\
$\bullet$ &  $\chi$ is {\em canonical} if $\chi$ is $X$-canonical or $Y$-canonical \\
$\bullet$ &  $\chi$ is {\em rainbow} if the colors of all the edges of $F$ are different and \\
$\bullet$ &  $\chi$ is {\em monochromatic} if the colors of all the edges of $F$ are the same.
\end{tabular}
\end{center}
\end{definition}

If $\chi$ is an edge-coloring of a graph $F$ and $G \subset F$, then $\chi\vert_G$ denotes the edge-coloring of $G$ obtained by restricting
$\chi$ to the edge-set of $G$. The following bipartite version of the canonical Ramsey theorem was proved in~\cite{KMV2}.

\begin{theorem} \label{canonical} For each $s>0$ there exists $t>0$ such that for any  edge-coloring $\chi $ of $G = K_{t,t}$,
 there exists  $K_{s,s} \subset G$ such that $\chi \vert_{K_{s,s}} $ is monochromatic or rainbow or canonical.
\end{theorem}

\begin{proof} Let $X$ and $Y$ be the parts of $G$ and let $S = \{y_1,y_2,\dots,y_{2s^2}\} \subset Y$. Let $W$ be the set of
vertices $x \in X$ contained in at least $s$ edges of the same color connecting $x$ with $S$. If $|W| > m := s^2{2s^2 \choose s}$, then
there is a set $Y' \subset S$ of size $s$ and a set $X' \subset W$ of size $s^2$ such that for every
$x \in X'$, the edges $xy$ with $y \in Y'$ all have the same color. In this case we
recover either a monochromatic $K_{s,s}$ or an $X'$-canonical $K_{s,s}$. Now suppose $|W| \leq m$.
For $x \in X_0 := X \backslash W$, let $C(x)$ be a set of $2s$ distinct colors on edges between $x$ and $S$.
By Lemma \ref{sunlemma}, if $|X_0| > (2s)! (s!m)^{2s}$, then there exists $X_1 \subset X_0$
such that $\{C(x) : x \in X_1\}$ is a Delta-system $\triangle$ of size $s!m$. Let $C = \mbox{core}(\triangle)$.
First suppose $|C| \geq s$. Then we have a set $X_2$ of at least $|X_1|/{2s^2 \choose s} \geq s!m{2s^2 \choose s} > s \cdot s!$ vertices in $X_1$ which each send $s$ edges with colors from $C$ into a fixed subset $Y_3$ of $S$ of size $s$.
 Since there are $s!$ orderings of the $s$ edges with colors from $C$ on each $x \in X_2$,
 this implies that for some set $X_3 \subset X_2$ of size $s$, the $K_{s,s}$ between $X_3$ and $Y_3$ is
$Y_3$-canonical. Finally, suppose $|C| < s$. Pick $C'(x) \subset C(x) \backslash C$ of size $s$ for  $x \in X_0$.
Since $|X_0| = s!m > s{2s^2 \choose s}$, we find a set $Y^* \subset S$ of  $s$ vertices as well as a set $X^* \subset X$ of $s$ vertices
$x \in X_0$ such that the edges between $x$ and $Y^*$ have colors from $C'(x)$. Since the sets $C'(x)$ are disjoint,
 this is a rainbow copy of $K_{s,s}$. \end{proof}


\subsubsection{List colorings in shadows}

In this section, we relate Theorem~\ref{canonical} to hypergraphs via the following definition.  Recall that if $H$ is a 3-graph and $e \in \partial H$, then the neighborhood $N_H(e)$ is the set of $v \in V(H)\setminus e$ such that $e \cup \{v\} \in H$.

\begin{definition} Let $H$ be a $3$-graph. For $G \subset \partial H$ and $e \in G$, let
\[ L_G(e) = N_H(e)\setminus V(G).\]
The set $L_G(e)$ is called the {\em list} of $e$ and the elements of $L_G(e)$ are called {\em colors}.
\end{definition}

Let $L_G = \bigcup_{e \in G} L_G(e)$ -- this is the set of colors in the lists of edges of $G$.

\begin{definition} A {\em list-edge-coloring of $G$} is a map $\chi : G \rightarrow L_G$ with $\chi(e) \in L_G(e)$
for all $e \in G$. List-edge-colorings $\chi_1,\chi_2 : G \rightarrow L_G$ are {\em disjoint} if
$\chi_1(e) \neq \chi_2(f)$ for all $e,f \in G$.
\end{definition}

If $\chi$ is an injection -- the coloring is rainbow -- then clearly $G^{+} \subset H$. We require one more definition.

\begin{definition} Let $H$ be a $3$-graph and $m \in \mathbb N$. An {\em $m$-multicoloring} of $G \subset \partial H$ is a family
of list-edge-colorings $\chi_1,\chi_2,\dots,\chi_m : G \rightarrow L_G$ such that $\chi_i(e) \neq \chi_j(e)$ for every $e \in G$
and $i\neq j$.
\end{definition}

A necessary and sufficient condition for the existence of
an $m$-multicoloring of $G$ is that all edges of $G$ have codegree at least $m$ in $H$.
We stress here that the definitions are all with respect to the fixed host $3$-graph $H$. The following result will be key to the proofs of Theorem \ref{main-path} and Theorem \ref{main-tree}.

\begin{theorem} \label{canon} Let $m,s \in \mathbb N$, let $H$ be a $3$-graph, and let $G = K_{t,t} \subset \partial H$.
Suppose $G$ has an $m$-multicoloring. If $t$ is large enough, then there exists $F = K_{s,s} \subset G$ such that $F$ has
either a rainbow list-edge-coloring or an $m$-multicoloring such that the colorings are pairwise disjoint, and each
coloring is monochromatic or canonical.
\end{theorem}

\begin{proof} Set $s=t_m/m^2$ and $t_m < t_{m-1} < \cdots < t_1 < t_0=t$ where Theorem \ref{canonical} with input $t_i$ has output $t_{i-1}$.
Pick a  color $c_1(e)$ on each edge $e\in G$ and apply Theorem \ref{canonical} to $G$.   We obtain a
rainbow, monochromatic or canonical subgraph $G_1$ of $G$ where $G_1=K_{t_1, t_1}$. If it is rainbow, then we are done, so assume it is monochromatic or canonical. For every $e\in G_1$, remove $c_1(e)$ from its list. Now pick another color on each edge of $G_1$ and repeat. We obtain subgraphs
$G_m \subset G_{m-1} \subset \cdots \subset G_1$ such that each $G_i$ is monochromatic or canonical where $G_i=K_{t_i, t_i}$ has parts $X_i, Y_i$.
In particular, each coloring of the $m$-multicoloring of $G$ restricted to $G_m$ is monochromatic or canonical.
\smallskip

Let us assume that we have $a$ monochromatic colorings, $b$ $X_m$-canonical colorings, and $c$ $Y_m$-canonical colorings of $G_m$ where $a+b+c=m$.
It suffices to ensure that these colorings are pairwise disjoint.
A  color $\chi(xy)$ in an $X_i$-canonical coloring of  $G_i$ cannot appear in a $Y_{i'}$-canonical coloring of $G_{i'}$ for $i'>i$ as $\chi(xy)$ was deleted from all edges incident to $x$ when forming $G_{i+1}$.  A similar statement holds with $X$ and $Y$ interchanged, so every $X_m$-canonical coloring of $G_m$ is disjoint from every  $Y_m$-canonical coloring of $G_m$.  The same argument shows that no color in a monochromatic coloring appears in a canonical coloring. It suffices to show that colors on different $X_m$-canonical colorings are disjoint (and the same for $Y_m$-canonical).
\smallskip

Let the $b$  $X_m$-canonical colorings be
$\chi_1,...,\chi_b$.
Construct an auxiliary graph $K$ with $V(K)=X_m$ where $xx'\in K$ if there exist $i\ne i'$ and a  color
$\alpha$ that is canonical for  $x$ in $\chi_i$ and  canonical for $x'$ in $\chi_{i'}$.
Then the maximum degree of $K$ is at most $b(b-1)$, so $K$ has an independent set of size
$s=t_m/m^2\ge |X_m|/b^2$. Let us restrict $X_m$ to this independent set. We repeat this procedure for $Y_m$ and finally obtain a subgraph $F=K_{s,s}$ with an $m$-multicoloring that satisfies the requirement of the theorem.
\end{proof}

\subsubsection{Random sampling}\label{randomktt}

Let $H$ be a triple system and $K \subset \partial H$. We say that $K$ has {\em out-codegree at least $d$}
if for every edge $e \in K$, there exist vertices $v_1,v_2,\dots,v_d \in V(H) \backslash V(K)$
such that $e \cup \{v_i\} \in H$ for $i \leq d$.

\begin{lemma} \label{kttlemma}
Let $m,t \in \mathbb N$, $\delta \in \mathbb R_+$ and $H$ be an $n$-vertex triple system.
Suppose that $F \subset \partial H$ and each $f \in F$ has $d_H(f)\ge m$.
If $|F| \geq \delta n^2$ and $n$ is large enough, then there is a complete bipartite graph $K_{t,t} \subset F$
with out-codegree at least $m$.
\end{lemma}

\begin{proof} Let $T$ be a random subset of $V(H)$ obtained by picking each vertex independently with probability $p=1/2$.
Let $G$ be the graph of all edges $f \in F$ with $f \subset T$ such that there exist $v_1,v_2,\dots,v_m \in V(H) \backslash T$
such that $f \cup \{v_i\} \in H$ for $1 \leq i \leq m$. Then
$$\mathbb E(|G|) \ge |F| p^{2}(1-p)^{m}\ge \frac{\delta}{2^{m+2}}n^{2}.$$
So there is a $T \subset V(H)$ with $|G|$ at least this large.
If $n$ is large enough, then Theorem \ref{kst} implies that there exists a complete bipartite graph $K \subset G \subset F$
with parts of size $t$. Due to the definition of $G$, the subgraph $K$ satisfies the requirements of the lemma.
\end{proof}

By Lemma \ref{kttlemma}, for a bipartite graph $G$, the problem of embedding $G^+$ in a triple system $H$ whose shadow is dense and $d$-full
is reduced to the problem of embedding $G^+$ in a subgraph of $H$ consisting of a $K_{t,t} \subset \partial H$ with out-codegree
at least $m$.

\begin{corollary}
Let $m \geq 2$ and $\delta > 0$, and let $H$ be an $m$-full triple system with $n$ vertices such that $|\partial H| \geq \delta n^2$.
If $n$ is large enough, then $G^+ \subset H$ for any bipartite graph $G$ with $|G| \leq m$.
\end{corollary}

\begin{proof}
Suppose $|V(G)| = t$. By Lemma \ref{kttlemma}, there exists a complete bipartite graph $K = K_{t,t} \subset \partial H$
such that $K$ has out-codegree at least $m$. Since $G \subset K$ and $|G| \leq m$, we may greedily embed the $m$ vertices of $V(G^+) \backslash V(G)$
in $V(H)$ to obtain $G^+ \subset H$.
\end{proof}

The corollary handles the case $m \geq |G|$, however, we will be applying the lemma with $m \ll |G|$,
and in particular with $m = \sigma(G^+)$ in Theorems \ref{main-path} and \ref{main-tree} for $r = 3$, and for this
we require Theorem \ref{canon}.

\section{Proof of Theorems}
In the next sections, we will show how to apply all the tools we have developed in section 4 to extremal problems for expansions.

\subsection{Proof of Theorem \ref{kst-expansion}}\label{proof-kst}

In this section we discuss how to apply Theorem \ref{kst} and Lemma \ref{fullsub} to obtain a short proof of
Theorem \ref{kst-expansion}. Since the proof of Theorem \ref{generalbipartite} is very similar, we omit that proof.

\medskip

{\bf Proof of Theorem \ref{kst-expansion}.}
The lower bound in Theorem \ref{kst-expansion} comes from a calculating the number of triangles in the norm
graph constructions in~\cite{ARS}. In particular, a $K_{s,t}$-free graph $G$ is constructed in~\cite{ARS}
which has $\Theta(n^{3 - 3/s})$ triangles. Taking $H$ to be the hypergraph whose vertex set is $V(G)$
and whose edge set is the set of triangles in $G$, we obtain a $K_{s,t}^+$-free triple system with
$\Theta(n^{3 - 3/s})$ edges. Precise details are presented in~\cite{KMV3} or~\cite{AS}.

\medskip

Next we show $\ex_3(n,K_{s,t}^+) = O(n^{3 - 3/s})$ for $t \geq s \geq 3$. For convenience, set $f(n) = \ex_2(n,K_{s,t})$
and $g(n) = \ex_2(n,K_{s-1,t})$. Theorem \ref{kst} gives bounds on these quantities.
By Lemma \ref{fullsub}, if $H$ is a triple system with no $K_{s,t}^+$, $H$ has an $(s + 1)(t + 1)$-full subgraph $H'$
where
\[ |H'| \geq |H| - (s + 1)(t + 1)|\partial H| = |H| - O(n^2).\]
If $|\partial H'| > f(n)$, then we find a $K_{s,t} \subset \partial H'$,
which can be greedily expanded to a $K_{s,t}^+ \subset H' \subset H$, a contradiction. Therefore $|\partial H'| \leq f(n)$.
If $d = |H|/2f(n)$, then there is a $d$-full subgraph $H''$ of $H'$ such that
\[ |H''| \geq |H'| - d|\partial H'| \geq \tfrac{1}{2}|H| - O(n^2).\]
Since $|\partial H''| \leq |\partial H'| \leq f(n)$, there is a vertex $v \in V(\partial H'')$
of degree at most $m = f(n)/n$ in $\partial H''$. If $N(v)$ is the neighborhood of $v$ in $\partial H''$,
then $N(v)$ induces a subgraph of $\partial H''$ with at most $g(m)$ edges, otherwise
we find a $K_{s-1,t}$ in $N(v)$, and adding $v$ we get a $K_{s,t} \subset \partial H'' \subset \partial H'$,
a contradiction. However, $H''$ is $d$-full, so $N(v)$ induces a subgraph of $H''$ of minimum degree at least
$d$. This shows $\tfrac{1}{2}dm \leq g(m)$, which implies $|H| = O(g(m)f(n)/m)$. By Theorem \ref{kst}, $f(n) = O(n^{2 - 1/s})$ which implies $m = O(n^{1 - 1/s})$. Since $g(m) = O(m^{2 - 1/(s - 1)})$ from Theorem \ref{kst}, we obtain
\[ |H| = O(g(m)f(n)/m) = O(n^{3 - 3/s}).\]
This completes the proof of Theorem \ref{kst-expansion}. \qed

\subsection{Proof of Theorem \ref{main-tree} for
$r \ge 4$}

In this section we show how to use Lemmas~\ref{kernelembed} and \ref{deltaconsequence} to prove Theorem \ref{main-tree} when $r \ge 4$.
Let us assume that $r \ge 4$, $G$ is a forest with $v$ vertices, $\sigma=\sigma(G)$, and $G^+ \not\subset H \subset {[n] \choose r}$. We are to prove an upper bound for $|H|$. Put $s=rv$ and let
\[ K = \{e \in \partial H : \d_H(e) \geq s\}.\]
  Lemma~\ref{kernelembed} implies that $G \not \subset K$, so the greedy algorithm yields $|K|\le (v-2)n$.  Let $V(K)=\{x_1, \ldots, x_n\}$ so that $d_K(x_i) \ge$ $d_K(x_{i+1})$. Let $L:=\{x_1, \ldots, x_p\}$ be the $p$ vertices of highest degree in $K$, where $p=n^{\epsilon}$ for some small positive $\epsilon$ depending only on $G$ and $r$. Using $|K| \le (v-2)n$ and an averaging argument yields
\begin{equation} \label{z} z:=d_{K}(x_{p+1}) \le
\frac{d_K(x_1) + \cdots + d_K(x_{p+1})}{p+1}
\le \frac{2|K|}{p+1} < \frac{2(v-2)n}{p}.\end{equation}
Consider the partition $H=H_1 \cup H_2$ in Lemma~\ref{deltaconsequence} and write
$H_2= H_3 \cup H_4 \cup H_5 \cup H_6$
where:
\begin{center}
\begin{tabular}{lp{5.5in}}
$\bullet$ & $H_3$ is the set of edges in $H_2$ with center outside $L$ \\
$\bullet$ & $H_4 \subset H$ is the set of edges meeting $L$ in at least two vertices \\
$\bullet$ & $H_5 \subset H_2\setminus H_3$ is the set of edges $f$ with $|f \cap L| =1$ and $d_{H_2\setminus H_3}(f \setminus L) \le \sigma-1$ \\
$\bullet$ & $H_6:=H_2\setminus (H_3 \cup H_4 \cup H_5)$.
\end{tabular}
\end{center}

We now show that $|H_3|, |H_4|, |H_6|$ are each $O(n^{r-1-\delta})$ for some $\delta>0$
while $|H_5| \le (\sigma-1){n-1 \choose r-1}$.

Counting edges of $H_3$ from their centers, an easy averaging argument, Lemma~\ref{deltaconsequence} and (\ref{z}) give
$$|H_3|\le \sum_{i=p+1}^n{d_K(x_i) \choose r-1}
\le \frac{\sum_i d_K(x_i)}{z}{z \choose r-1} = O(n^{r-1}/p^{r-2})= O(n^{r-1-r\epsilon}).$$
Clearly
$$|H_4| \le {p \choose 2}{n \choose r-2}= O(n^{r-2+2\epsilon}).$$
Counting the edges of $H_5$ from their portion outside $L$ immediately gives
$$|H_5| \le (\sigma-1){n-p \choose r-1}.$$
For $A \in {L \choose \sigma}$  let $\cB_A:=\{B: \{a\} \cup B \in H \hbox{ for all $a \in A$}\}$ and
$$H_A:=\{f \in H: a \in A, B \in \cB_A, \{a \}
\cup B=f \}.$$
Each edge $f \in H_6$ can be written in the form
$f=\{a \} \cup B$ where $a \in L$, $B \cap L=\emptyset$, such that there are at least $\sigma$ edges in $H\setminus H_3$ containing $B$. Moreover, because none of these $\sigma$ edges are in $H_3$ their last point is in $L$. Consequently, $H_6 \subset \bigcup_{A\in {L\choose \sigma}} H_A$.

Given a copy of $G^+$, let $Y$ be a set of $\sigma$ vertices of this copy that intersect each edge of $G^+$ in exactly one vertex (here is where we use the definition of $\sigma$). Let $\cC$ be the $(r-1)$-graph obtained from $G^+$ by removing the set $Y$. In other words, every edge of $G^+$ gets shrunk by removing the unique vertex of $Y$ from it. Since $G^+ \not\subset H$, we conclude that $\cC\not\subset \cB_A$ for each $A$. Consequently, $|\cB_A| \le \ex_{r-1}(n, \cC) = O(n^{r-2})$ where the bound holds from an easy application of Lemma~\ref{fullsub}. Therefore
$$|H_6| \le \sum_{A \in {L \choose \sigma}}|H_A|
\le  \sum_{A \in {L \choose \sigma}}|A||\cB_A|
\le \sigma{p \choose \sigma}O(n^{r-2})= O(n^{r-2+\epsilon \sigma}).$$

\subsection{Proof of Theorems \ref{main-path}, \ref{main-tree} and \ref{k2t-expansion} for $r = 3$}
\label{template}

The proofs of Theorems \ref{main-path}, \ref{main-tree} and \ref{k2t-expansion} for $r=3$ all
use the same template, which we describe as follows.
Let $G$ be a tree or a cycle or a graph with $\sigma(G^+) = 2$. Let $H$ be an $n$-vertex triple system with at least
$(\sigma(G^+) - 1 + \varepsilon){n \choose 2}$ edges, where $\varepsilon > 0$. If we can show $G^+ \subset H$
when $n$ is large enough, then we have $\ex_3(n,G^+) \sim (\sigma(G^+) - 1){n \choose 2}$.

\medskip

By Lemma \ref{fullsub}, $H$ has a $\sigma$-full subgraph $H'$ with at least $\varepsilon {n \choose 2}$ edges. Let $|V(G)| = k$. If $|\partial H'| < \tfrac{\varepsilon}{2k}{n \choose 2}$, then by Lemma \ref{fullsub}, $H'$ contains a $2k$-full subgraph $H''$. In particular, $\partial H''$ has minimum degree at least $k$, and therefore contains $G$. Now every edge of this copy of $G$ has degree at least $2k$ in $H''$,
and therefore we may greedily expand $G$ to $G^+$ in $H'$, as required. Next
we consider the case $|\partial H'| \geq \frac{\varepsilon}{2k}{n \choose 2}$.

\medskip

By Lemma \ref{kttlemma}, for any $t$
there exists $n_0$ such that $n > n_0$, then
there exists a $K = K_{t,t} \subset \partial H'$ each edge of which has out-codegree at least $\sigma$.
So if $e \in K$, there exist vertices $v_1(e),v_2(e),\dots,v_{\sigma}(e) \in V(H') \backslash V(K)$ such that
$e \cup \{v_i(e)\} \in H'$ for $1 \leq i \leq \sigma$. In the language of list colorings, we have $\sigma$ list colorings
$\chi_1,\chi_2,\dots,\chi_{\sigma}$ of $K$, defined by $\chi_i(e) = v_i(e)$ for $e \in K$. By Theorem \ref{canon},
for any $s$ there exists $n_1$ such that if $n > n_1$, there exists a complete bipartite graph $J \subset \partial H'$ with parts
$X$ and $Y$ of size $s$ such that the restriction of $\chi_i$ to $J$ is canonical, monochromatic or rainbow, for $1 \leq i \leq \sigma$.

\medskip

The remainder of the proof consists in showing that we can embed $G^+$ in $H'$ using the colorings of $J$.
The proof of the embedding is generally technical for Theorem \ref{main-tree} and Theorem \ref{k2t-expansion},
so we focus on the special case of embedding an expansion of a quadrilateral i.e.
Theorem \ref{main-path} for $k = 4$. In this case $\sigma = 2$, and we let $s = 4$, and we consider the cases one by one:

\begin{center}
\begin{tabular}{lp{5.5in}}
$\bullet$ & {\em Some $\chi_i$ is rainbow.} Then $C_4^+ \subset J^+ \subset H'$. \\
$\bullet$ & {\em $\chi_1$ and $\chi_2$ are monochromatic.} Suppose  $\chi_1(e) = v_1$ for $e \in J$ and $\chi_2(e) = v_2$ for $e \in J$.
Let $e_1,e_2$ and $e_3,e_4$ form vertex-disjoint paths of length two in $J$. Then
$e_1 \cup \{v_1\},e_2 \cup \{v_2\},e_3 \cup \{v_2\},e_4 \cup \{v_1\}$ is a $C_4^+$ in $H'$.\\
$\bullet$ & {\em $\chi_1$ and $\chi_2$ are both canonical.} Suppose $\chi_1$ and $\chi_2$ are both $X$-canonical, the
other cases are similar. Then pick any quadrilateral $\{e_1,e_2,e_3,e_4\} \subset J$,
where $e_1 \cap e_2 \subset X$ and $e_3 \cap e_4 \subset X$. By definition there exist distinct vertices $v_1,v_2$ and $w_1,w_2$
such that $\chi_1(e_1) = v_1$, $\chi_2(e_2) = v_2$, $\chi_1(e_3) = w_1$ and $\chi_2(e_4) = w_2$. Then
$e_1 \cup \{v_1\},e_2 \cup \{v_2\},e_3 \cup \{w_1\},e_4 \cup \{w_2\}$ is a $C_4^+$ in $H'$.\\
$\bullet$ & {\em $\chi_1$ is canonical and $\chi_2$ is monochromatic.} Suppose $\chi_1$ is $X$-canonical. Let
$e_1,e_2,e_3,e_4$ form a path of length four in $J$ with $e_2 \cap e_3 \subset Y$. Then there exist
distinct vertices $v_1,v_2,v_3$ such that $\chi_1(e_2) = v_2$, $\chi_1(e_3)= v_3$ and
$\chi_2(e_1) = \chi_2(e_4) = v_1$. Now $e_1 \cup \{v_1\},e_2 \cup \{v_2\},e_3 \cup \{v_3\},e_4 \cup \{v_1\}$ is a $C_4^+$ in $H'$.
\end{tabular}
\end{center}

In all cases, $C_4^+ \subset H'$, and we conclude $\ex_3(n,C_4^+) \sim {n \choose 2}$. The
cases of embedding $G^+$ when $G$ is a cycle or a tree follow the same template, except that
the cases are substantially more complicated, and for trees the embedding
depends on the structure of the tree relative to crosscuts. \qed

\section{Open problems}

When $G$ is a forest or a cycle, or $\sigma(G^+) = 2$, the main theorems in the survey show $\ex_r(n,G^+) \sim (\sigma(G^+) - 1){n - 1 \choose r}$.
It is an interesting general question to determine for which $r$-graphs $F$ one has $\ex_r(n,F) \sim (\sigma(F) - 1){n - 1 \choose r}$,
and furthermore, for which $F$ the extremal constructions all have a small transversal, or can be transformed
by deleting few edges to an $r$-graph with a small transversal. A general problem is to determine those graphs $G$ with  $\ex_r(n,G^+) \sim (\sigma(G^+) - 1){n \choose r - 1}$.
In the case $r = 3$, the proof template described in Section~\ref{template} breaks the argument into two broad cases depending on the size of the shadow.  If the shadow is very small, then the greedy procedure allows us to embed a rather large class of expansions, including those with treewidth at most two. It is convenient to use the following characterization of graphs with treewidth at most two to see this: start with an edge, and given any edge $xy$ in the current graph, add a new vertex $z$ and the two edges $xz$ and $yz$.  This leads us to pose the following problem:

\begin{problem} \label{tw} If $G$ has treewidth two,
determine whether $\ex_3(n,G^+) \sim (\sigma(G) - 1){n \choose 2}$.
\end{problem}

By Theorem \ref{k2t-expansion}, if $\sigma(G^+) = 2$ then $\ex_3(n,G^+) \sim (\sigma(G) - 1){n \choose 2}$. The next case $\sigma(G^+)=3$ is wide open.
\begin{problem}
Determine the order of magnitude of $\ex_3(n,G^+)$ when $G$ is a graph with $\sigma(G^+) = 3$.
\end{problem}

The results of~\cite{KMV3} show $\ex_3(n,K_{3,t}^+)/n^2 \rightarrow \infty$ as $t \rightarrow \infty$,
so $\ex_3(n,G^+)$ probably depends substantially on the structure of $G$. However, one may
ask if $\ex_3(n,G^+) = O(n^2)$ when $\sigma(G^+) = 3$.

\medskip

A general problem is to determine for which $r$-graphs $F$ one has $\ex_r(n,F) = O(n^{r-1})$. Theorem \ref{generalbipartite} shows that if $F = G^+$
and $G$ is a bipartite graph with $\ex_2(n,G) = O(n^{\varphi})$ where $\varphi = \frac{1}{2}(\sqrt{5}- 1)$,
then $\ex_3(n,G^+) = O(n^2)$, and this gives $\ex_3(n,Q_3^+) = O(n^2)$ and also $\ex_3(n,G^+) = O(n^2)$
when $G$ is a bipartite graph such that every vertex in some part of $G$ has degree at most two. On the other hand,
we saw that $\ex_3(n,G^+)$ is not quadratic in $n$ if in every 3-coloring of $G$, every pair of color classes
induces a cycle, or if $|G| \geq 3|V(G)| - 5$. We pose the following problem. An {\em acyclic coloring} of a graph $G$
is a proper vertex-coloring of $G$ such that any pair of color classes induce a forest.

\begin{problem}
Determine whether $\ex_3(n,G^+) = O(n^2)$ for every graph $G$ with an acyclic 3-coloring.
\end{problem}

In particular, every wheel graph $W_k$ with $k + 1$ vertices, with $k$ even, has an acyclic 3-coloring,
and the above problem is open even for $W_4$. We may phrase the above problem in the language of graphs:
if an $n$-vertex graph has a superquadratic number of triangles, does it contain every graph with a bounded number of
vertices and an acyclic 3-coloring, and in particular, every wheel graph $W_k$ with $k$ even?

\medskip

Recall that a graph is 2-degenerate if each of its subgraphs has minimum degree at most two. Every graph with treewidth two is 2-degenerate, but not vice versa, for example the graph obtained from $K_4$ by subdividing an edge is 2-degenerate but has treewidth three.   Burr and Erd\H os~\cite{BE} conjectured that for every 2-degenerate bipartite graph $G$, $\ex_2(n,G) = O(n^{3/2})$.
Using Theorem \ref{generalbipartite}, this would show $\ex_3(n,G^+) = O(n^2)$ for every 2-degenerate bipartite graph $G$.
It is therefore perhaps reasonable to expect that every 2-degenerate graph $G$ has $\ex_3(n,G^+) = O(n^2)$ -- note that these
graphs all have an acyclic 3-coloring.
\begin{problem}
If $G$ is a 2-degenerate graph, is $\ex_3(n,G^+) = O(n^2)$?
\end{problem}
If $G$ is a bipartite graph such that all the vertices in one part have degree at most two, then it is known that $\ex(n,G) = O(n^{3/2})$ by the method of Dependent Random Choice~\cite{FoxS},
and therefore in this case $\ex_3(n,G^+) = O(n^2)$.

\medskip

Finally, we mention that one can also consider expansions of hypergraphs: if $G$ is an $s$-graph
then the $r$-uniform expansion $G^+$ of $G$ is obtained by adding a set of $r - s$ vertices to each edge of $G$,
with disjoint sets for distinct edges. Recently, F\"{u}redi and Jiang~\cite{FJs} extended Theorem~\ref{main-tree}
to a large class of $r$-uniform expansions of $s$-graphs forests. Also, the first author and Stading~\cite{MSt} determined the asymptotics
for $\ex_r(n, P_t(s)^+)$ for $r \ge 2s$, where
$P_t(s)$ is the $2s$-graph consisting of $t$ edges
$e_1, \ldots, e_t$ where consecutive edges intersect in $s$ points and nonconsecutive edges are pairwise disjoint.

\medskip

In general, determining $\ex_r(n, G^+)$ for an $s$-graph $G$ is a wide open problem with almost no general results known.
An interesting case is when $G$ is a tight tree. A {\em tight tree} is a hypergraph with edge set $\{e_1,e_2,\dots,e_k\}$ such that for every $i \geq 2$, there exists $j < i$ such that
$e_i \cap \bigcup_{h = 1}^{i - 1} e_h \subset e_j$. In particular, if all the $e_i$ have size two, then we have the usual definition of a graph tree. An important conjecture of Kalai~\cite{FF},
extending the well-known Erd\H{o}s-S\'{o}s conjecture for trees in graphs~\cite{ErdosSos}, states that if $T$ is an $r$-uniform tight tree with $k$ vertices, then $\ex_r(n,T) \leq \frac{k - r}{r} {n \choose r - 1}$. This conjecture remains open in all but a very few cases,
and the proof of the Erd\H{o}s-S\'{o}s conjecture for graphs was only recently claimed. The following general question is raised for expansions:

\begin{problem}
Determine $\lim_{n \rightarrow \infty} \frac{\ex_r(n,T^+)}{{n \choose r - 1}}$ when $T$ is the shadow of an $r$-uniform tight tree.
\end{problem}

Problem \ref{tw} asks whether the limit is $\sigma(T^+) - 1$ in the case $r = 3$, since the shadow of a 3-uniform tight tree is a graph of treewidth two.


\begin{thebibliography}{99}


\bibitem{AK} R. Ahlswede and L. H. Khachatrian, The complete intersection theorem for systems of finite sets, European J. Combin., 18 (1997), pp. 125–-136.

\bibitem{AKS} N. Alon, M. Krivelevich, B. Sudakov,  Tur\'{a}n numbers of bipartite graphs and related Ramsey-type questions.
Combinatorics, Probability and Computing 12 (2003), 477--494.

\bibitem{AP} N. Alon, O. Pikhurko, personal communication

\bibitem{ARS} N. Alon,  L.  R\'{o}nyai, T.  Szab\'{o},  Norm-graphs: variations and applications, J. Combin. Theory, Ser. B 76 (1999), 280--290.

\bibitem{AS} N. Alon,  C. Shikhelman,   Triangles in $H$-free graphs, http://arxiv.org/abs/1409.4192.


\bibitem{BDE} B. Bollob\'as, D.E. Daykin, P. Erd\H os, Sets of independent edges of a hypergraph, Quart. J. Math. Oxford Ser. (2) 27 (1976)
25--32.


\bibitem{Brown}  W. G. Brown,  On graphs that do not contain a Thomsen graph, Canad. Math. Bull. 9 (1966), 281--285.

\bibitem{BE}  S. Burr, P. Erd\H{o}s, On the magnitude of generalized Ramsey numbers for graphs,
in: J. Bolyai, editor, Infinite and finite sets I, Vol. 10, Colloquia Mathematica Societatis, NorthHolland,
Amsterdam, 1975, pp. 214--240.


\bibitem{BK} N. Bushaw, N. Kettle,  Tur\'an Numbers for Forests of Paths in Hypergraphs. arXiv:1303.5022.

\bibitem{CLW} W. Y. C. Chen, J. Liu, L. X. Wang,  Families of sets with intersecting clusters. SIAM J. Discrete Math. 23 (2009), no. 3, 1249–1260.

\bibitem{C} F. R. K. Chung, Unavoidable stars in 3-graphs, J. Comb. Theory (A) (1983), 35, 252--262.


\bibitem{FC} F. R. K.  Chung, P. Frankl,  The maximum number of edges in a 3-graph not containing a given star. Graphs Combin. 3 (1987), no. 2,
111--126.


\bibitem{CK}  R. Cs\'ak\'any, J.  Kahn,  A homological approach to two problems on finite sets. J. Algebraic Combin. 9 (1999), no. 2, 141--149.

\bibitem{dF} D. de Caen, Z.  F\"uredi,  The maximum size of 3-uniform hypergraphs
not containing a Fano plane. J. Combin. Theory Ser. B 78 (2000),
no. 2, 274--276.

\bibitem{DEF}  M. Deza, P. Erd\H os, P. Frankl,  Intersection properties of systems of finite sets, Proc. London Math. Soc. (3)
36 (1978), 369--384.

\bibitem{DE}
R. A. Duke, P.  Erd\H os,  Systems of finite sets having a common intersection. In: Proceedings, 8th
S-E Conf. Combinatorics, Graph Theory and Computing, pp. 247--252. 1977

\bibitem{ErdosSos} P. Erd\H{o}s, Some problems in graph theory, Theory of Graphs and Its Applications,
M. Fiedler, Editor, Academic Press, New York, 1965, pp. 29--36.


\bibitem{E1964} P. Erd\H os, On extremal problems of graphs and generalized graphs,
Israel J. Math, 2, (1964) 183--190.

\bibitem{E1965} P. Erd\H os, A problem on independent $r$-tuples, Ann. Univ. Sci. Budapest. 8 (1965) 93–95.




\bibitem{E1970}
P. Erd\H os,  On the graph theorem of Tur\' an. (Hungarian) Mat.
Lapok 21 (1970), 249--251 (1971).


\bibitem{E} P. Erd\H{o}s,
Problems and results in graph theory and combinatorial analysis,
Proceedings of the Fifth British Combinatorial Conference (Univ. Aberdeen, Aberdeen, 1975),
pp. 169--192, Congressus Numerantium, No. XV, Utilitas Math., Winnipeg, Man., 1976.


\bibitem{ET} P. Erd\H os,
Problems and results on graphs and hypergraphs: similarities and
differences. Mathematics of Ramsey theory, 12--28, Algorithms
Combin., 5, Springer, Berlin, 1990.




\bibitem{EG} P. Erd\H os, T. Gallai, On maximal paths and circuits of graphs, Acta Math. Acad. Sci. Hungar. 10 (1959) 337–356.

\bibitem{EKR} P. Erd\H{o}s, C. Ko, R. Rado,  Intersection theorems for systems of finite sets, The Quarterly Journal of Mathematics. Oxford. Second Series (1961) 12, 313--320.

\bibitem{ER} P. Erd\H{o}s, R.  Rado,  A combinatorial theorem. J. London Math. Soc. 25 (1950), 249--255.

\bibitem{ES66}
 P. Erd\H os,  M. Simonovits,  A limit theorem in graph theory, Studia Sci. Math. Hungar. 1 (1966), 51--57.

\bibitem{ES} P. Erd\H os, M.  Simonovits, Supersaturated graphs and
hypergraphs, Combinatorica 3 (1983), no. 2, 181--192.



\bibitem{ErdosSimCube} P. Erd\H{o}s, M. Simonovits: Cube supersaturated graphs and related problems, Progress in Graph Theory, (Waterloo, Ont., 1982), 203--218, Academic Press, Toronto, ON, 1984. (eds Bondy and Murty)


\bibitem{ErdosStone}  P. Erd\H{o}s, A. H. Stone, On the structure of linear graphs. Bull. Amer. Math. Soc. 52 (1946), 1087--1091.





\bibitem{FoxS} J. Fox, B. Sudakov, Dependent Random Choice, Random Struct. Alg. 38 (2011), 1--32.


\bibitem{Frankl1977} P. Frankl,  On families of finite sets no two of which intersect in a singleton. Bull. Austral. Math. Soc.
17 (1977), no. 1, 125--134.

\bibitem{Frankl1978} P. Frankl,  An extremal problem for 3-graphs. Acta Math. Acad. Sci. Hung. (1978), 32, 157--160.


\bibitem{Frankl1983} P. Frankl,  An extremal set theoretical characterization of some Steiner systems, Combinatorica
3, 193-199 (1983)

\bibitem{Frankl1990} P. Frankl, Asymptotic solution of a Tur\'an-type problem,
 Graphs and Combinatorics, 6 (1990), 223-227.

\bibitem{Frankl2013} P. Frankl,
Improved bounds for Erd\H os' matching conjecture,
J. Combin. Theory Ser. A 120 (2013), no. 5, 1068–1072.


\bibitem{FF} P. Frankl, Z. F\" uredi,  Exact solution of some Tur\'{a}n-type problems. J. Combin. Theory Ser. A 45 (1987), no. 2, 226--262.

\bibitem{FLM} P. Frankl, T. \L uczak, K. Mieczkowska, On matchings in hypergraphs, Electron. J. Combin. 19 (2012), Paper 42, 5 pp.


 \bibitem{FR}
P. Frankl, V.  R\"odl,  Hypergraphs do not jump, Combinatorica 4 (1984), no. 2-3, 149--159.



\bibitem{FBCC} Z.  F\"uredi,  Tur\'an type problems. Surveys in combinatorics, 1991 (Guildford, 1991), 253--300,
London Math. Soc. Lecture Note Ser., 166, Cambridge Univ. Press,
Cambridge, 1991.

\bibitem{F-k2t} Z. F\"{u}redi, New asymptotics for bipartite Tur\'{a}n numbers, J. Combin. Theory, Ser. A, 75 (1996), pp. 141--144.

\bibitem{Furedi}  Z. F\"{u}redi,
 Linear trees in uniform hypergraphs. European J. Combin. Theory Ser. A 35 (2014), 264--272.

\bibitem{Furedidelta} Z. F\"uredi, On finite set systems whose every intersection is a kernel of a star, Discrete Math. 47 (1983), 129--132

\bibitem{FJ} Z. F\"{u}redi, T. Jiang,  Hypergraph
Tur\'an numbers of linear cycles. arXiv:1302.2387


\bibitem{FJs} Z. F\"{u}redi, T. Jiang, manuscript  in preparation.

\bibitem{FJS} Z. F\"{u}redi, T. Jiang, R. Seiver,  Exact solution of the hypergraph Tur\'{a}n problem for $k$-uniform linear paths, To appear in Combinatorica (2013).


\bibitem{FMP} Z. F\"uredi, D. Mubayi, O. Pikhurko Quadruple systems with independent neighborhoods, J. Combin. Theory, Ser. A 115 (2008), no. 8, 1552--1560.

\bibitem{FPS}  Z. F\"uredi, O. Pikhurko, M. Simonovits,
 The Turan Density of the Hypergraph $\{abc,ade,bde,cde\}$, Electron. J. Combin., 10 (2003) 7pp.

\bibitem{FPS2}  Z. F\"uredi, O. Pikhurko, M. Simonovits,
On Triple Systems with Independent Neighborhoods, to appear, Comb.
Prob. \& Comput.

\bibitem{FS} Z. F\"uredi, M. Simonovits, Triple systems not
containing a Fano configuration, submitted.


\bibitem{GLS}
R. Gu, X. Li, Y.  Shi,  Tur\'{a}n numbers of vertex disjoint cycles. arXiv:1305.5372.


\bibitem{HLS} H. Huang, P. Loh, B. Sudakov, The size of a hypergraph and its matching number, Combin. Probab. Comput. 21 (2012)
442–450.


\bibitem{JPY}  T. Jiang, O. Pikhurho,  Z. Yilma, Set Systems without a Strong Simplex, SIAM J. Discrete Math, 24 (2010) 1038-1045.

\bibitem{KM} P. Keevash, D. Mubayi, Stability results for cancellative hypergraphs,
 J. Combin. Theory Ser. B  92  (2004),  no. 1, 163--175.

\bibitem{KM2} P.  Keevash, D. Mubayi,  Set systems without a simplex or a cluster. Combinatorica 30 (2010), no. 2, 175–200.

\bibitem{KMW} P. Keevash, D. Mubayi, R. Wilson, Set systems with no singleton intersection, SIAM J. Discrete Math 20 (2006), no. 4, 1031--1041.

\bibitem{KS} P. Keevash, B. Sudakov, The exact Tur\'an number of
the Fano plane, to appear, Combinatorica.

\bibitem{KS2} P. Keevash, B. Sudakov, On a hypergraph Tur\'an problem of Frankl,
to appear, Combinatorica.

\bibitem{K} J. Kim, Regular subgraphs of uniform hypergraphs,  arXiv:1502.02177 [math.CO]


\bibitem{KMV}  A. Kostochka, D. Mubayi, J.  Verstra\"ete,  Tur\'an problems and shadows I: paths and cycles, to appear, J.
Combin. Theory Ser. A

\bibitem{KMV2}  A. Kostochka, D. Mubayi, J. Verstra\"ete,  Tur\'an problems and shadows  II: trees, submitted.

\bibitem{KMV3}  A. Kostochka, D. Mubayi, J. Verstra\"ete,  Tur\'an problems and shadows  III: bipartite graphs, to appear, SIAM J. Discrete Math

\bibitem{KST} T. K\"{o}vari, V. T. S\'{o}s, P. Tur\'{a}n, P,
On a problem of K. Zarankiewicz. Colloquium Math. 3, (1954) 50--57.


\bibitem{LM} T. \L uczak, K. Mieczkowska, On
Erd\H os' extremal problem on matchings in hypergraphs,  J.
 Combin. Theory Ser. A 124, (2014)
 178--194.


\bibitem{MS}
T. S. Motzkin, E. G. Straus, Maxima for graphs and a new proof of a theorem of Tur\'an,
 Canad. J. Math. 17, 1965, 533--540.


\bibitem{Mubayi} D. Mubayi,  A hypergraph extension of Tur\'{a}n's theorem.
J. Combin. Theory, Ser. B, 96 (2006), no. 1, 122--134.

\bibitem{M2} D. Mubayi, An intersection theorem for four sets, Adv. in Math., 215 (2007) no. 2, 601--615.

\bibitem{MRam1} D. Mubayi, R. Ramadurai,
 Simplex stability, Combin. Probab. \& Comput. 18 (2009), no. 3, 441--454.

\bibitem{MRam} D. Mubayi, R. Ramadurai,  Set systems with union and intersection constraints. J. Combin. Theory Ser. B 99 (2009), no. 3, 639–642.

\bibitem{MRodl} D. Mubayi, V. R\"odl, On the Tur\'an number of triple systems,
 J. Combin. Theory Ser. A 100, (2002), 136--152.

\bibitem{MSt} D. Mubayi, R. Stading, Ph.D thesis of R. Stading, University of Illinois at Chicago, (2014).

\bibitem{MV1} D. Mubayi, J. Verstra\"ete,  Minimal paths and cycles in set-systems, European J.  Combin. 28 (2007) no.6, 1681--1693


\bibitem{MV} D. Mubayi, J. Verstra\"{e}te,  Proof of a conjecture of Erd\H{o}s on triangles in set-systems. Combinatorica 25 (2005), no. 5, 599--614.

\bibitem{MV2} D. Mubayi, J. Verstra\"ete,  Two-regular subgraphs of hypergraphs. J. Combin. Theory Ser. B 99 (2009), no. 3, 643–655.

\bibitem{Pikhurko} O. Pikhurko,  Exact Computation of the Hypergraph Tur\'{a}n Function for Expanded Complete 2-Graphs.
J. Combin. Theory  Ser. B 103 (2013) 220--225.

\bibitem{Sstab} M. Simonovits,
 A method for solving extremal problems in graph theory, stability problems, 1968
 Theory of Graphs (Proc. Colloq., Tihany, 1966) pp. 279--319

\bibitem{Sim} M. Simonovits, How to solve a Tur\'{a}n type extremal graph problem? (linear decomposition). Contemporary trends in discrete mathematics (Stirín Castle, 1997), 283--305, DIMACS Ser. Discrete Math. Theoret. Comput. Sci., 49, Amer. Math. Soc., Providence, RI, 1999.


\bibitem{S}  V. T. S\'os, Some remarks on the connection of graph theory, finite geometry and block designs.
In: Proc. Combinatorial Conf., pp. 223--233 Rome 1976


\bibitem{T} P. Tur\'{a}n, On an extremal problem in graph theory. Mat. Fiz. Lapok 48 (1941), 436--452.

\bibitem{W} R. M. Wilson,
The exact bound in the Erd\H{o}s-Ko-Rado theorem,
Combinatorica (4) (1984), 247--257.

\bibitem{Z} K. Zarankiewicz, Problem of P101, Colloq. Math., 2 (1951), p. 301.







\end{thebibliography}
\end{document}